\documentclass[11pt]{amsart}

\usepackage{latexsym, amssymb, amsmath, amscd, amsfonts, amsthm, enumitem}
\usepackage{url}
\usepackage{xcolor}
\usepackage{kotex}
\usepackage{ytableau}
\usepackage{geometry}
\usepackage{setspace}
\usepackage[all]{xy}

\numberwithin{equation}{section}
\newtheorem{thm}{Theorem}[section]

\newtheorem*{defn}{Definition}
\newtheorem*{ex}{Example}
\newtheorem{lem}[thm]{Lemma}
\newtheorem{prop}[thm]{Proposition}

\newtheorem*{note}{Notation}
\newgeometry{left=4cm, right=3cm, top=4cm, bottom=3cm}

\begin{document}

\title{Spin and Representations}
\author{Wonmyeong Cho}
\address{(Wonmyeong Cho) Department of Mathematical Sciences, Seoul National University,
\newline\indent Seoul, South Korea}
\email{mangoringo@snu.ac.kr}

\begin{abstract}
We derive the representation theory of $SU(2)$ from the expository theory of Lie groups and Lie algebras. Based on this, the mathematics of non-relativistic quantum mechanics of a spin $\frac{1}{2}$ particle are described from a representation-theoretic perspective, and are extended to many particle systems.
\end{abstract}

\maketitle


\bigskip\section{Lie Groups and Lie Algebras}

We begin by evaluating the Lie algebras for some matrix groups, which will give us an easier method to construct the desired representations at hand.

\begin{defn}\normalfont
A \textbf{Lie group} $G$ is a topological group that is also a manifold, such that the group operation $G \times G \rightarrow G$ is differentiable.
The \textbf{Lie algebra} of $G$ is defined to be the tangent space of $G$ at the identity element $e$, and is denoted as $\mathfrak{g}$.
\end{defn}

In physics, we are mainly interested in the matrix groups $U(n), SO(n), SU(n)$ as representations of the group of symmetries of a certain object.
For example, $SO(3)$ represents the group of 3-dimensional rotations.
Since the Lie algebra is a linearization of the Lie group, it is much easier to study than say, the corresponding neighborhood of the identity that we approximated.
\medskip\newline\indent Throughout, we assume $G$ is a finite-dimensional matrix group, hence multiplication in $\mathfrak{g}$ is defined as matrix multiplication. This, however, is not the operation that makes $\mathfrak{g}$ into an algebra, and in general, $\mathfrak{g}$ need not be closed under matrix multiplication.

\begin{lem}\normalfont
Let $X \in \mathfrak{g}, \epsilon \in \mathbb{R}$.
For sufficiently small $\epsilon$, there exists a group element of the form $$1 + \epsilon X + \sum_{k=2}^{\infty} c_{k}\epsilon^{k}X^{k}$$
where $c_{k}$ are arbitrary real coefficients.
\end{lem}

\begin{note}\normalfont
By use of the big-O notation, denote such element as $1 + \epsilon X + \mathcal{O}(\epsilon^{2})$.
\end{note}

\begin{prop}\normalfont
Let $g \in G, X, Y \in \mathfrak{g}, t \in \mathbb{R}$. Then,
\begin{enumerate}[label=(\roman*)]
\item $e^{tX} \in G$
\item $XY-YX \in \mathfrak{g}$
\end{enumerate}
\end{prop}

\begin{proof} Observe that $\mathfrak{g}$ is a real vector space, thus it suffices to prove (i) for $X$.
\newline For sufficiently large $N \in \mathbb{N}$, there exists a group element of the form $1 + \frac{1}{N}X + \mathcal{O}(\frac{1}{N^{2}})$;
take the limit $N \rightarrow \infty$.
\smallskip\newline For (ii), define $\pi_{g}(t) = ge^{tX}g^{-1}$. The tangent vector at $t=0$ is an element of $\mathfrak{g}$, thus
$$\frac{d}{dt}\pi_g(t) \biggr\rvert_{t=0} = gXg^{-1} \in \mathfrak{g}$$
Take $g = e^{tY}$, then $e^{tY}Xe^{-tY} \in \mathfrak{g}$; differentiate at $t=0$ to obtain the desired result.
\end{proof}

\begin{defn}\normalfont
We define the \textbf{commutator} of $X, Y \in \mathfrak{g}$ as $[X,Y] = XY-YX$.
\end{defn}

\bigskip We are now ready to compute the Lie algebras of $U(n)$ and $SU(n)$.

\begin{lem}\normalfont
As vector spaces, $\text{dim}\, G = \text{dim}\, \mathfrak{g}$.
\end{lem}

This follows from a basic result in differential topology that tangent spaces of a manifold have the same dimension as the manifold itself.
A proof is in [GP], page 9.

\begin{thm}\normalfont
The Lie algebras of $U(n)$ and $SU(n)$ are given as follows:
$$\mathfrak{u}(n) = \{X \in M_{\mathbb{C}}(n) \mid X+X^{\dagger} = 0\}$$
$$\mathfrak{su}(n) = \{X \in M_{\mathbb{C}}(n) \mid X+X^{\dagger} = 0, \text{tr}\, X = 0\}$$
\end{thm}

\begin{proof}
We start by showing that $\text{dim}\, U(n) = n^{2}$.
\newline\smallskip The constraint $UU^{\dagger} = 1$ reads $\displaystyle\sum_{k}u_{ik}\bar{u}_{jk}=\delta_{ij}$;
\newline $i=j$ yields $1$ constraint in $\mathbb{R}$, and $i \neq j$ yields $1$ constraint in $\mathbb{R}$, $\mathbb{I}$, each.
\newline\smallskip Thus, $\text{dim}\, U(n) = 2n^{2} - n - 2 \times \displaystyle\frac{n(n-1)}{2} = n^{2}$.
\newline Next, given $X \in \mathfrak{u}(n), t \in \mathbb{R}$, we have $e^{tX}(e^{tX})^{\dagger} = e^{tX}e^{tX^{\dagger}} = 1$.
\newline\smallskip Differentiating at $t=0$ yields $X+X^{\dagger} = 0$. Define $C = \{X \in M_{\mathbb{C}}(n) \mid X+X^{\dagger} = 0\}$.
\newline Counting the degrees of freedom of $X$, we deduce $\text{dim}\, C = n^{2}$, thus $C = \mathfrak{u}(n)$.
\newline\smallskip For $\mathfrak{su}(n)$, we use the fact that $SU(n) = U(n) \cap SL(n)$ thus $\mathfrak{su}(n) = \mathfrak{u}(n) \cap \mathfrak{sl}(n)$.
\newline Given $X \in \mathfrak{sl}(n)$, $\text{det}\, (e^{X}) = e^{\text{tr}\, X} = 1$, thus $\text{tr}\, X = 0$.
\newline The space of traceless matrices and $SL(n)$ both have dimension $2n^{2}-2$; we are done.
\end{proof}


\section{$\mathfrak{su}(2)$ and Spin Representations}

We are particularly interested in $SU(2)$ in non-relativistic quantum mechanics as the symmetric group of the Hilbert space of a particle carrying spin.
What is spin, you ask? For now, I'll say that it is a special kind of angular momentum that is an instrinsic property of the particle.
In that sense, we expect $SU(2)$ to behave similarly to the group of rotations, $SO(3)$.
By deriving the fundamental commutator relations for the spin operators, we can finally begin to dig into the theory of spin.
Let's get into it!

\medskip First though, to be able to draw this connection between quantum mechanics and representations, we need to introduce some basic physical language.

\begin{defn}\normalfont
The \textbf{wavefunction} of a particle is any function that satiesfies the equation
$$\hat{H}|\psi\rangle = E|\psi\rangle$$
This eigenvalue equation is called the \textbf{Time-independent Schr\"odinger equation} (TISE for short) and is determined by the \textbf{Hamiltonian} operator $\hat{H}$ of the particle.
The eigenvalues $E$ of the equation are called the \textbf{energies} of the particle.
Any operator that commutes with the Hamiltonian is said to be \textbf{compatible}.
The associated eigenspace of the equation is called the \textbf{Hilbert space} of the particle.
\end{defn}

In essence, what we are trying to do is, instead of solving the Schr\"odinger equation directly (which is, for most of the time, near impossible),
try to find some set of compatible operators that give us information on what the Hilbert space looks like.
We have no interest in what form the wavefunctions take; they can be functions, they can be vectors, they can be cats and dogs, as long as they spit out an energy eigenvalue.

Just remember this; anything we do here has an analogue in the TISE picture, and solving the TISE is equivalent to knowing the action of the Hamiltonian on the Hilbert space.
Physicists seem to go to far lengths to solve that innocent-looking eigenvalue equation for some random particle whose existence is not even known; what we are about to do now is one of the cleaner methods of doing so.

\medskip From \textbf{Theorem 1.4}, we see that $\mathfrak{su}(2)$ is generated by
$$X_{1} = \frac{1}{2}\begin{pmatrix}0 & -i \\ -i & 0 \end{pmatrix}, X_{2} = \frac{1}{2}\begin{pmatrix}0 & -1 \\ 1 & 0 \end{pmatrix}, X_{3} = \frac{1}{2}\begin{pmatrix}-i & 0 \\ 0 & i \end{pmatrix}$$

\begin{defn}\normalfont
The \textbf{Pauli matrices} are $\sigma_{j} = 2iX_{j}$, and the \textbf{spin operators} are \newline $S_{j} = iX_{j} = \frac{1}{2}\sigma_{j}$, for $j = 1,2,3$;
the \textbf{ladder operators} are $S_{\pm} = S_{1} \pm iS_{2}$.
\end{defn}

The spin operators satisfy the following commutation relations:
$$[S_{i}, S_{j}] = i\epsilon_{ijk}S_{k} \indent [S_{3}, S_{\pm}] = \pm S_{\pm} \indent [S_{+}, S_{-}] = 2S_{3}$$

\begin{defn}\normalfont
The \textbf{complexification} of a Lie algebra $\mathfrak{g}$ spanned by basis $\mathcal{B}$ over $\mathbb{R}$
is defined as the space spanned by $\mathcal{B}$ over $\mathbb{C}$, and is denoted as $\mathfrak{g}_{\mathbb{C}}$.
\end{defn}

$\mathfrak{su}(2)_{\mathbb{C}} = \text{span}_{\mathbb{C}}\, \{X_{1}, X_{2}, X_{3}\} = \text{span}_{\mathbb{R}}\, \{X_{1}, X_{2}, X_{3}, S_{1}, S_{2}, S_{3}\} = \text{span}_{\mathbb{C}}\, \{S_{+}, S_{-}, S_{3}\}$

\begin{defn}\normalfont
Let $\mathfrak{g}$ be a Lie algebra and $V$ a vector space. A \textbf{Lie algebra representation} is a map ${\Phi}: \mathfrak{g} \rightarrow \mathfrak{gl}(V)$ between Lie algebras, such that, for all $a, b \in \mathbb{R}$, $X, Y \in \mathfrak{g}$:
\begin{enumerate}[label=(\roman*)]
\item $\Phi(aX+bY) = a\Phi(X) + b\Phi(Y)$
\item $\Phi([X, Y]) = [\Phi(X), \Phi(Y)]$
\end{enumerate}
\end{defn}

Before putting this definition to use, we first classify the irreps of $U(1)$.

\begin{lem}[Quantization Condition]\normalfont
Every irrep of $U(1)$ can be written as \newline $\phi_{n}: u \mapsto u^{n}$, where $n$ is an integer.
\end{lem}

\begin{proof}
Elements of $U(1)$ are simply complex numbers such that $u\bar{u} = 1$, i.e. $u = e^{i\theta}$.
Since $U(1)$ is abelian, all its irreps have dimension 1, by Schur's lemma.
\newline Differentiating $\phi$ in terms of $\theta$, we have
$$\frac{d}{d\theta}\phi(e^{i\theta}) = \displaystyle\lim_{\epsilon \to 0} \frac{\phi(e^{i(\theta+\epsilon)})-\phi(e^{i\theta})}{\epsilon} = \phi(e^{i\theta}) \displaystyle\lim_{\epsilon \to 0} \frac{\phi(e^{i\epsilon})-\phi(1)}{\epsilon} = n\phi(e^{i\theta})$$
where $n = \phi'(0)$; thus $\phi(e^{i\theta}) = e^{in\theta}$, and $n$ is bounded by the homomorphism condition $e^{2\pi in} = 1$. The result follows.
\end{proof}

It must be commented on that it is not immediately obvious how this connects to the quantization derived from, say, boundary conditions of a wave function.
Well, what exactly do the boundary conditions signify? They are manifestations of the conditions that the representation of the associated Hamiltonian in Hilbert space must satisfy.
In other words, the periodicity of the wave function correlates directly to the cyclicity of $U(1)$.
We demonstrate this by example of the spin operator that lives in $SU(2)$.

\smallskip\indent Before doing so, we introduce a useful tool that allows us to go back and forth from Lie groups and Lie algebras.

\begin{thm}\normalfont Let $G$ be a simply connected Lie group with Lie algebra $\mathfrak{g}$.
\begin{enumerate}[label=(\roman*)]
\item Let $\phi: G \rightarrow GL(V)$ be a Lie group representation. Then, $\phi$ induces a Lie algebra representation $\Phi$ with the property that $\phi(e^X) = e^{\Phi(X)}$ for all $X \in \mathfrak{g}$, given by
$$\Phi(X) = \frac{d}{dt}\phi(e^{tX})\biggr\rvert_{t=0}$$
\item Every Lie algebra representation $\Phi: \mathfrak{g} \rightarrow \mathfrak{gl}(V)$ arises in this manner.
\end{enumerate}
\end{thm}

The proof is left to [H], page 60, Theorem 3.28 for (i) and page 119, Theorem 5.6 for (ii).
Since $SU(2)$ is homeomorphic to the 3-sphere (write out the constraints for $U^{\dagger}U=UU^{\dagger}=1$), it is simply connected and thus we may apply the above result.

\begin{thm}[The Spin Representation]\normalfont
For every nonnegative half-integer $s$, there is a unique irrep of $SU(2)$ of dimension $2s+1$, which induces a representation of $\mathfrak{su}(2)$ of the same dimension, given by
$$\Phi_{s}(S_{3}) = \begin{pmatrix}s & 0 & \dots & 0 \\ 0 & s-1 & \dots & 0 \\ \vdots & \vdots & \ddots & \vdots \\ 0 & 0 & \dots & -s \end{pmatrix},$$
$$\Phi_{s}(S_{+}) = \begin{pmatrix}0 & b_{s-1} & 0 & \dots & 0 \\ 0 & 0 & b_{s-2} & \dots & 0 \\ \vdots & \vdots & \vdots & \ddots & \vdots \\ 0 & 0 & 0 & \dots & b_{-s} \\ 0 & 0 & 0 & \dots & 0 \end{pmatrix},
\Phi_{s}(S_{-}) = \begin{pmatrix}0 & 0 & \dots & 0 & 0 \\ b_{s-1} & 0 & \dots & 0 & 0 \\ 0 & b_{s-2} & \dots & 0 & 0 \\ \vdots & \vdots & \ddots & \vdots & \vdots \\ 0 & 0 & \dots & b_{-s} & 0 \end{pmatrix}$$
\smallskip\newline where $b_{m} = \sqrt{s(s+1) - m(m+1)}$, for $m = -s, \dots, s-1$ are chosen for convenience.
\end{thm}

\begin{proof}
Let $\phi: SU(2) \rightarrow GL(V)$ be an irrep with finite dimension.
\smallskip\newline We begin with $S_3$; notice that $e^{2i\theta S_{3}} = \begin{pmatrix}e^{i\theta} & 0 \\ 0 & e^{-i\theta}\end{pmatrix}$,
\smallskip\newline thus the subgroup it generates in $SU(2)$ is isomorphic to $U(1)$.
\newline Choose a basis where $\phi(e^{2i\theta S_{3}})$ is diagonal, and using \textbf{Theorem 2.2}(i),
\newline we see that $\Phi(S_{3})$ is a diagonal matrix consisting of half integers.
\newline Label the diagonal entries $a_{i}$, where $i = 0, 1, \dots, n$ in non-decreasing order.
\smallskip\newline Here's the catch: given an eigenvector $x_{i}$ of $\Phi(S_{3})$ with eigenvalue $a_{i}$, we have
$$\Phi(S_{3})(\Phi(S_{+})x_{i}) = (a_{i}+1)(\Phi(S_{+})x_{i}), \Phi(S_{3})(\Phi(S_{-})x_{i}) = (a_{i}-1)(\Phi(S_{-})x_{i})$$
$$[\Phi(S_{3}), \Phi(S_{+})]x_{i} + \Phi(S_{+})\Phi(S_{3})x_{i} = \Phi(S_{3})\Phi(S_{+})x_{i} = (1+a_{i})\Phi(S_{+})x_{i}$$
The calculation is similar for $\Phi(S_{-})$. This tells us that, since the set of eigenvalues $\{a_{i}\}$ is bounded by $[a_0, a_n]$, $\Phi(S_{+})x_{n}$ and $\Phi(S_{-})x_{0}$ are zero vectors.
\smallskip\newline \textbf{Claim:} V is spanned by $\mathcal{B}_{+} = \{x_{0}, \Phi(S_{+})(x_{0}), (\Phi(S_{+}))^{2}(x_{0}), \dots\}$.
\smallskip\newline Note that the set is finite since it terminates eventually. It suffices to show that the space spanned by $\mathcal{B}_{+}$ is closed under action of $\mathfrak{su}(2)_{\mathbb{C}} = \text{span}_{\mathbb{C}}\, \{S_{+}, S_{-}, S_{3}\}$,
for then it follows that $\text{span}\, \mathcal{B}_{+}$ is a $SU(2)$-linear subspace of $V$, thus equals $V$.
Here, we use the fact that every $U \in SU(2)$ is of the form $U = e^{X}$ for some $X \in \mathfrak{su}(2)$ .
\smallskip\newline The case is closed for $\Phi(S_{3})$ and $\Phi(S_{+})$. For $\Phi(S_{-})$, induct on the power of $\Phi(S_{+})$:
\newline $\Phi(S_{-})(x_{0}) = 0$ for the base case; for the general case, observe
$$\Phi(S_{-})\Phi(S_{+}) = \Phi(S_{+})\Phi(S_{-}) - [\Phi(S_{+}), \Phi(S_{-})] = \Phi(S_{+})\Phi(S_{-}) - 2\Phi(S_{3})$$
Apply both sides to $(\Phi(S_{+}))^{m-1}(x_{0})$ to obtain the desired result.
\smallskip\newline The claim allows us to write $a_{i} = a_{0}+i$. Now, consider the case $n = 0$.
\newline Here, $x_{0} = x_{n}$, thus $V = \text{span}\, \{x_{0}\}$, and we have $\Phi(S_{+}) = \Phi(S_{-}) = \Phi(S_{3}) = 0$.
\newline It follows that $\phi$ is trivial, since $\phi(e^{X}) = e^{\Phi(X)} = 1$ for all $X \in \mathfrak{g}$.
\smallskip\newline Let $V$ have dimension $n+1$ for any positive integer $n$.
\newline Composition of $\phi$ with the determinant representation $\mathfrak{D}: GL(V) \rightarrow GL_{1}(\mathbb{C}) \simeq \mathbb{C}$ given by $X \mapsto \text{det}\, X$ yields a 1-dimensional representation $\mathfrak{D} \circ \phi: SU(2) \rightarrow GL_{1}(\mathbb{C})$.
\newline This representation is irreducible, being of dimension 1, thus is trivial.
\newline Since $\text{det}\, \phi(e^{X}) = e^{\text{tr}\, \Phi(X)} = 1$, we see that $\Phi(X)$ is traceless for all $X \in \mathfrak{g}$.
\smallskip\newline In particular, $\displaystyle\sum_{i=0}^{n}a_{i} = 0$, thus $a_{i} = -a_{n-i}$, which gives us $S_{3}$.
\smallskip\newline Constructing $S_{\pm}$ from the eigenvector condition, we are done.
\end{proof}

From the above proof, notice that $a_{n} = \frac{n}{2}$ can be any nonnegative half integer, and the dimension of the representation is given by $2a_{n}+1$ (labeled `$s$' in the statement).
The factors $b_{m}$ are arbitrary constants that melt into the eigenvector; they are chosen such that the theory of spin mimics that of angular momentum.

\begin{defn}\normalfont
From the statement of \textbf{Theorem 2.3}, this Lie algebra representation characterizes the Hilbert space of a particle with \textbf{spin} $s$, and thus is called the is called the \textbf{spin $s$ representation}.
The eigenvectors $x_{i}$ are called \textbf{states}, and we denote them as bras $\langle sm|$ and kets $|sm\rangle$ for column and row vectors, respectively, where $m$ is the associated eigenvalue of $\Phi_{s}(S_{3})$ and is called the \textbf{magnetic quantum number}.
\smallskip\newline An \textbf{observable} is a Hermitian operator that has $V$ as its eigenspace.
It is called so because it has real eigenvalues and thus will have a physical manifestation.
Any wave function associated to an observable compatible with the Hamiltonian of this particle will be a linear combination of these states.
\end{defn}

We now have that $\Phi(S_{+})|sm\rangle \propto |s(m+1)\rangle, \Phi(S_{-})|sm\rangle \propto |s(m-1)\rangle$.
\newline Finally, we define a new observable, called the spin-squared operator.
$$S^{2} := S_{1}^{2} + S_{2}^{2} + S_{3}^{2} = S_{+}S_{-} - S_{3} + S_{3}^{2}$$
Notice that it shares simultaneous eigenvectors as $S_{3}$, hence $[S^{2}, S_{3}] = 0$.
\smallskip\newline\indent We summarize our results below:

\begin{thm}\normalfont
The action of the spin operators on the states are given as follows:
$$S_{3}|sm\rangle = m|sm\rangle, S^{2}|sm\rangle = s(s+1)|sm\rangle$$
$$S_{\pm}|sm\rangle = \sqrt{s(s+1) - m(m\pm1)}|s(m\pm1)\rangle$$
\end{thm}\medskip


\section{Many Particle Systems and Clebsch-Gordan Coefficients}

We are now ready to construct the theory for many particle systems.
First, consider what space is spanned by the states of a two-particle system, each with spin $s_{1}$, $s_{2}$.
Each pair of two individual state $|s_{1}m_{1}\rangle|s_{2}m_{2}\rangle$ is a state of the system, thus our space has as basis all such states;
thus naturally, we construct the two-particle Hilbert space as the tensor product of the two single particle Hilbert spaces.

\smallskip\indent The generalization to any system consisting of a finite number of particles is straightforward.
To avoid messy notation and ellipses, I will proceed with two-particle systems, but keep in mind that all these concepts generalize to any finite number of particles.

\begin{defn}\normalfont
Let a system $X$ be composed of particles 1, 2 with spin $s_{1}$, $s_{2}$, respectively.
We denote the spin operators of each system by $S_{i}^{(1)}$, $S_{i}^{(2)} (i = 1,2,3,\pm)$, respectively.
\newline The total spin operators are defined as $S_{i} = S_{i}^{(1)} \otimes 1 + 1 \otimes S_{i}^{(2)} (i = 1,2,3, \pm)$, 
\newline and the total spin-squared operator is then calculated as
$$S^{2} = S_{1}^{2} + S_{2}^{2} + S_{3}^{2} = S^{2^{(1)}} \otimes 1 + 1 \otimes S^{2^{(2)}} + 2(S_{1}^{(1)} \otimes S_{1}^{(2)} + S_{2}^{(1)} \otimes S_{2}^{(2)} + S_{3}^{(1)} \otimes S_{3}^{(2)})$$
\end{defn}

\begin{prop}\normalfont
The total spin operators satisfy the commutation relations
$$[S_{i}, S_{j}] = i\epsilon_{ijk}S_{k} \indent [S_{3}, S_{\pm}] = \pm S_{\pm} \indent [S_{+}, S_{-}] = 2S_{3}$$
and therefore induce a basis of states $|sm\rangle$, where $s$ runs from $|s_{1} - s_{2}|$ to $s_{1} + s_{2}$.
\end{prop}

\begin{proof}
The commutation relations follow from simple algebra.
\smallskip\newline We desire to find simultaneous eigenstates of $S^{2}$ and $S_{3}$.
\newline Observe that for given state $|s_{1}m_{1}\rangle|s_{2}m_{2}\rangle$, we have
$$S_{3}|s_{1}m_{1}\rangle|s_{2}m_{2}\rangle = (m_{1}+m_{2})|s_{1}m_{1}\rangle|s_{2}m_{2}\rangle$$
thus we have $m = m_{1} + m_{2}$, and it suffices to consider the degeneracy of the associated eigenspace for each $m = -(s_{1}+s_{2}), \dots, s_{1}+s_{2}$.
\smallskip\newline Assume $s_{1} \geq s_{2}$ wlog; the number of ordered pairs $(m_{1}, m_{2})$ summing to $m$ with the constraints $|m_{1}| < s_{1}, |m_{2}| < s_{2}$ gives the degeneracy level.
\newline Counting the pairs, we see that the degeneracy level is $2s_{2}+1$ for $|m| \leq s_{1}-s_{2}$, and decreases by $1$ as $|m|$ increases by $1$ for each $s_{1}-s_{2}, \dots, s_{1}+s_{2}$.
\newline Since the states are independent and the total number is
 $$(2(s_{1}-s_{2})+1)(2s_{2}+1) + 2\displaystyle\sum_{n = 1}^{2s_{2}} n = (2s_{1}+1)(2s_{2}+1),$$
these states fill the whole of the Hilbert space spanned by $\{|s_{1}m_{1}\rangle|s_{2}m_{2}\rangle\}$.
\end{proof}

 From the above proof, observe that for each $s = |s_{1}-s_{2}|, \dots, s_{1}+s_{2}$, the total spin operator induces a spin $s$ representation.
By \textbf{Theorem 2.3}, these are irreps of $SU(2)$, thus we have effectively decomposed the tensor product of two irreps of $SU(2)$.
\smallskip\newline Since this is the main result in terms of representation theory, we formally state it below:

\begin{thm}\normalfont
Let $V_{s_{1}}, V_{s_{2}}$ be the $SU(2)$-linear spaces associated to the spin $s_{1}, s_{2}$ \newline representations, respectively.
Then, the decomposition of $V_{s_{1}} \otimes V_{s_{2}}$ is given by
$$V_{s_{1}} \otimes V_{s_{2}} = \displaystyle\bigoplus_{s = |s_{1}-s_{2}|}^{s_{1}+s_{2}} V_{s}$$
where $V_{s}$ is the $SU(2)$-linear space associated to the spin $s$ representation.
\end{thm}

Now that we have done the two-particle case, having obtained a new basis in the form $|sm\rangle$,
we may view the tensor product Hilbert space as a direct sum of familiar ones; this allows us to apply the construction of total spin operators inductively.
\smallskip\newline\indent We conclude with a discussion of the transformation of the $|sm\rangle$ basis to the $|s_{1}m_{1}\rangle|s_{2}m_{2}\rangle$ basis.
This topic is of great significance in physics since it gives a relation between the microscopic and macroscopic scales.

\begin{defn}\normalfont
We define the \textbf{Clebsch-Gordan coefficients} $C_{m_{1}m_{2}m}^{s_{1}s_{2}s}$ as follows:
$$|sm\rangle = \displaystyle\sum_{(m_{1}, m_{2})} C_{m_{1}m_{2}m}^{s_{1}s_{2}s}|s_{1}m_{1}\rangle|s_{2}m_{2}\rangle$$
Note that $C_{m_{1}m_{2}m}^{s_{1}s_{2}s} = 0$ unless $m_{1}+m_{2}=m$, considering action by $S_{3}$.
\end{defn}

I will guide you through an elementary example of calculating these coefficients, then tell you the general method.

\begin{ex}\normalfont
Take two spin $\frac{1}{2}$ particles; denote $|\frac{1}{2}\frac{1}{2}\rangle = |\!\uparrow\rangle, |\frac{1}{2}\,\text{-}\frac{1}{2}\rangle = |\!\downarrow\rangle$.
\newline These are the so-called spin `up' and `down' states, applicable to any spin $\frac{1}{2}$ particle \newline (in particular, all known fermions).
\smallskip\newline We desire to find the transformation $\{|\!\uparrow\uparrow\rangle, |\!\uparrow\downarrow\rangle, |\!\downarrow\uparrow\rangle, |\!\downarrow\downarrow\rangle\} \rightarrow \{|11\rangle, |10\rangle, |1-1\rangle, |00\rangle\}$.
\newline Starting with the top state of highest spin, we have $|11\rangle = |\!\uparrow\uparrow\rangle$.
$$S_{-}|11\rangle = \sqrt{2}|10\rangle = (S_{-}^{(1)} \otimes 1 + 1 \otimes S_{-}^{(2)})|\!\uparrow\uparrow\rangle = |\!\downarrow\uparrow\rangle + |\!\uparrow\downarrow\rangle$$
$$S_{-}|10\rangle = \sqrt{2}|1-1\rangle = (S_{-}^{(1)} \otimes 1 + 1 \otimes S_{-}^{(2)})\frac{1}{\sqrt{2}}(|\!\downarrow\uparrow\rangle + |\!\uparrow\downarrow\rangle) = \sqrt{2}|\!\downarrow\downarrow\rangle$$
Solve for $|00\rangle$ using the orthogonality condition $\langle10|00\rangle = 0$.
Conventionally, we choose the sign so that the state with the highest $m_{1}$ (assuming $s_{1} \geq s_{2}$) is positive.
$$|11\rangle = |\!\uparrow\uparrow\rangle, |10\rangle = \frac{1}{\sqrt2}(|\!\uparrow\downarrow\rangle + |\!\downarrow\uparrow\rangle), |1-1\rangle = |\!\downarrow\downarrow\rangle, |00\rangle = \frac{1}{\sqrt{2}}(|\!\uparrow\downarrow\rangle - |\!\downarrow\uparrow\rangle)$$
\end{ex}

For the general case, start with the topmost state $|(s_{1}+s_{2})(s_{1}+s_{2})\rangle = |s_{1}s_{1}\rangle|s_{2}s_{2}\rangle$, as with the example.
Climb down the ladder until you reach the bottom rung; symmetry of the ladder indicates that the coefficients are symmetric about the $m = 0$ rung:
$$\langle s(m-1)|S_{-}|sm\rangle = \langle s(-m+1)|S_{+}|s-m\rangle$$
which saves you half the calculations. Moving onto the total spin $s_{1}+s_{2}-1$ states, use the orthogonality condition on $|(s_{1}+s_{2}-1)(s_{1}+s_{2}-1)\rangle$ and normalize, putting the plus sign on the `upper' state.
Climb down the ladder once more to retreive all the coefficients for this spin.
For each time you step down a spin, you get one more orthogonality condition to work with; rinse and repeat.

\smallskip The general Clebsch-Gordan coefficients can be derived in explicit form using this method, since we already know our starting point at the topmost rung.
It is not a pleasing calculation to work out, but if you want to see the general form, refer to [B], page 171, and the calculations leading up to (2.41).
\newline\indent Here is a particular general case of this form, where we have a system consisting of a particle with arbitrary spin and a particle with spin $\frac{1}{2}$.

\begin{thm}\normalfont
Let $s_{1}$ be any half-integer, $s_{2} = \frac{1}{2}$. Then, $s = s_{1}\pm\frac{1}{2}$, and
$$|sm\rangle = \sqrt{\frac{s_{1}\pm m+1/2}{2s_{1}+1}}\,|s_{1}(m-\tfrac{1}{2})\rangle|\tfrac{1}{2}\tfrac{1}{2}\rangle \pm\sqrt{\frac{s_{1}\mp m+1/2}{2s_{1}+1}}\,|s_{1}(m+\tfrac{1}{2})\rangle|\tfrac{1}{2}\,\text{-}\tfrac{1}{2}\rangle.$$
\end{thm}

\begin{proof}
Have $|sm\rangle = A|s_{1}(m-\tfrac{1}{2})\rangle|\tfrac{1}{2}\tfrac{1}{2}\rangle + B|s_{1}(m+\tfrac{1}{2})\rangle|\tfrac{1}{2}\,\text{-}\tfrac{1}{2}\rangle$.
\smallskip\newline Write $S^{2} = S_{+}S_{-} + S_{3}^{2} - S_{3} = S_{-}S_{+} + S_{3}^{2} + S_{3}$.
\newline We desire to find $A, B$ such that $|sm\rangle$ is an eigenstate of $S^{2}$ with eigenvalue $s(s+1)$.
\newline Setting $x = (s_{1}+\tfrac{1}{2})^{2}, y=\sqrt{x-m^{2}}$ for convenience, we have
$$S^{2}|s_{1}(m-\tfrac{1}{2})\rangle|\tfrac{1}{2}\tfrac{1}{2}\rangle = (x+m)|s_{1}(m-\tfrac{1}{2}\rangle|\tfrac{1}{2}\tfrac{1}{2}\rangle + y|s_{1}(m+\tfrac{1}{2}\rangle|\tfrac{1}{2}\,\text{-}\tfrac{1}{2}\rangle$$
$$S^{2}|s_{1}(m+\tfrac{1}{2})\rangle|\tfrac{1}{2}\,\textbf{-}\tfrac{1}{2}\rangle = (x-m)|s_{1}(m+\tfrac{1}{2}\rangle|\tfrac{1}{2}\,\text{-}\tfrac{1}{2}\rangle + y|s_{1}(m-\tfrac{1}{2}\rangle|\tfrac{1}{2}\tfrac{1}{2}\rangle$$
$$S^{2}|sm\rangle = s(s+1)|sm\rangle = ((x+m)A+yB)|s(m-\tfrac{1}{2})\rangle|\tfrac{1}{2}\tfrac{1}{2}\rangle + (yA+(x-m)B)|s(m+\tfrac{1}{2})\rangle|\tfrac{1}{2}\,\text{-}\tfrac{1}{2}\rangle$$
$$(x+m)A+yB = s(s+1)A,\,\, yA + (x-m)B = s(s+1)B$$
The two conditions should be one and the same; with some more algebra, we have
$$(x-s(s+1))^{2}-x = 0,\, \text{or}\,\, s = s_{1}\pm\tfrac{1}{2}$$
All that is left to do is normalize and set the sign of $A$ to be positive.
\end{proof}

\smallskip We end on a physical note by introducing the concept of entanglement.

\begin{defn}\normalfont
A linear combination of states is said to be \textbf{entangled} if it cannot be written as a single tensor product of states.
\end{defn}

What this means in our physical context is that, once we obtain a certain result for particle A (i.e. we `observe' it),
the state of particle B has been determined, regardless of any external conditions concerning the two particle system.
The two particles cannot exist independently of each other, hence the term `entanglement'.
Let this sink in for a second. If Alice observes particle A on the sun, then Bob sees particle B in a fixed state on the Earth from the very time frame that Alice observed particle A.
\smallskip\newline\indent How can this be? The system cannot send information faster than light, but somehow the particles `knows' instantly what has happened to its partner.
This phenomenon is the basis for the famous \textbf{EPR paradox}, and arose a feverous discussion on causality of events and hidden variable theory, but that is beyond what I can hope to cover here.

\smallskip Let's get back on track. I'll give you a simple example.

\begin{ex}\normalfont
Recall the system in the previous example consisting of two spin $\tfrac{1}{2}$ particles.
\smallskip It is easy to see that the states with $m=0$ are entangled:
$$|10\rangle = \frac{1}{\sqrt2}(|\!\uparrow\downarrow\rangle + |\!\downarrow\uparrow\rangle),\, |00\rangle = \frac{1}{\sqrt{2}}(|\!\uparrow\downarrow\rangle - |\!\downarrow\uparrow\rangle)$$
We already know that the tensor product must be of in a form such that $m_{1}+m_{2} = m$, but the forms written above are the exact linear combination of such states.
\smallskip\newline Suppose we measure the state $|10\rangle$ with $S_{3}^{(1)}$, which returns $m_{1}$:
\newline we obtain $\tfrac{1}{2}$ with probability $\tfrac{1}{2}$ and $\text{-}\tfrac{1}{2}$ with probability $\tfrac{1}{2}$.
\newline Then, for each such measurement on particle $1$, we know $m_{2}$ without measuring it!
\end{ex}

For any wavefunction consisting of a linear combination of states $|sm\rangle$, by using the Clebsch-Gordan coefficients, we can decompose all of them into linearly independent basis states of the form $|s_{1}m_{1}\rangle|s_{2}m_{2}\rangle$.
By measuring these states accordingly with the spin operators in the Hilbert space of each particle, we may analyze the entanglement problem associated with the original state.
\newline\indent We refer to measuring such a state as `collapsing the wave function': once we restrict the component $|s_{1}m_{1}\rangle$, we are simply left with a linear combination of $|s_{2}m_{2}\rangle$ states, and the wavefunction is no longer entangled.

\end{document}